\documentclass[reqno,11pt]{amsart}

\usepackage{amssymb,amsthm}
\usepackage{amsmath}
\usepackage{amsfonts}
\usepackage{dsfont}

\usepackage[a4paper,  margin=3.6cm]{geometry}

\usepackage[active]{srcltx}
\makeatletter\@addtoreset{equation}{section}\makeatother

\newtheorem{theorem}{Theorem}[section]
\newtheorem{corollary}[theorem]{Corollary}

\newtheorem{proposition}[theorem]{Proposition}
\newtheorem{assumption}[theorem]{Assumption}
\newtheorem{definition}[theorem]{Definition}
\newtheorem{remark}[theorem]{Remark}
\numberwithin{equation}{section}

\title[Infinite populations as complex systems]{Infinite populations of migrants as complex systems: self-regulation}

\author{ Yuri  Kozitsky}

\address{Instytut Matematyki, Uniwersytet Marii Curie-Sk{\l}odowskiej, Plac Marii Curie-Sk{\l}odowskiej 1, 20-031 Lublin, Poland}
\email{jkozi@hektor.umcs.lublin.pl}

\begin{document}


\begin{abstract}

A model is proposed and studied describing an infinite population of
point migrants arriving in and departing from $X\subseteq
\mathds{R}^d$, $d\geq 1$. Both these acts occur at random with
state-dependent rates. That is, depending on their geometry the
existing migrants repel and attract the newcomers, which makes the
population a complex system. Its states are probability measures on
an appropriate configuration space, and their evolution $\mu_0 \to
\mu_t$ is obtained by solving the corresponding Fokker-Planck
equation. The main result is the conclusion that this evolution of
states preserves their sub-Poissonicity, and hence a local
self-regulation (suppression of clustering) takes place due to the
inter-particle repulsion -- no matter of how small range. Further
possibilities to study the proposed model with the help of this
result are also discussed.

\end{abstract}

\maketitle

\section{Introduction and Setup}
\subsection{Infinite populations in noncompact habitats}
``A complex system is any system featuring a large number of
interacting components (agents, processes, etc.) whose aggregate
activity is nonlinear (not derivable from the summations of the
activity of individual components) and typically exhibits
hierarchical self-organization under selective pressures." This
typical description\footnote{see
http://www.informatics.indiana.edu/rocha/publications/complex/csm.html}
 (definition) of a complex system clearly demonstrates that self-regulation belongs to its
crucial properties arising from interactions between the
constituents. Large systems of living entities are always named in
this context. The development of the mathematical theory of such
systems is a challenging task of modern applied  mathematics
\cite{BB,B,Neuhauser}. This includes also modeling populations as
configurations of point entities placed in some continuous habitat,
see \cite{BP1,BP3,Mu,Ova,O}, in particular, those where the dynamics
amounts to the appearance and disappearance of their members. For
such models, a key question of their theory is how the
appearance/disappearance of a given entity is related to its
interaction with the rest of the population and what might be the
outcome of such interaction.

In simple models, the considered populations are finite and have no
structure. The state space of such a population is usually taken as
the set of nonnegative integers $\mathds{N}_0 :=\mathds{N}\cup
\{0\}$. In such a case, in state $n\in \mathds{N}_0$ the population
consists of $n$ entities, which is its complete characterization.
Then the only observed result of the interplay between the
appearance and disappearance is the evolution of the number of
entities in the population. Of course, possible interactions in this
case can be taken into account only indirectly. The theory of such
models goes back to works by A. Kolmogorov and W. Feller, see
\cite[Chapter XVII]{Feller} and, e.g., \cite{Banas,Ri}, for a more
recent account of the related concepts and results. Therein, the
time evolution of the probability of having $n$ entities is obtained
by solving the Kolmogorov equation with a tridiagonal infinite
matrix on the right-hand side. Its entries are expressed in terms of
birth and death rates, $\lambda_n$ and $\mu_n$, respectively.
Possible influence of the entities on each other is reflected in the
dependence of theses rates on $n$. If the increase of $\lambda_n$
and $\mu_n$ is controlled by affine functions of $n$, the solution
of the Kolmogorov equation is given by a stochastic semigroup, see,
e.g., \cite{Banas} and the literature quoted in that work.

A population becomes a complex system if it is given an inner
structure closely related to interactions of its constituents -- now
explicitly taken into consideration. An inner structure of a given
population can be set by assigning a trait to each of its
constituents. Usually, interactions are binary, dependent on the
traits of the interacting entities. The set of all such traits, $X$,
is mostly a locally compact topological space. It is called
\emph{habitat} being interpreted as the set of spatial locations of
the entities. Its topology is employed to determine the (finite)
collection of entities interacting with a given one. Namely, it is
supposed that the traits of the entities interacting with the entity
with trait $x$ lie in a compact neighborhood of $x$. This determines
the \emph{local structure} of the population. In order to
qualitatively distinguish between local and global aspects of its
dynamics, one has to assume that $X$ is noncompact. If the whole
population is finite, it is contained in a compact subset $Y\subset
X$ -- a finite sum of compact sets containing the trait of each
single entity. If the population `stays' in this $Y$ forever, then
$Y$ is an actual habitat of the population, in contrast to our
assumption. If in the course of evolution the population disperses
beyond each compact subset of $X$, it can be characterized as
\emph{developing}. In such systems, boundary effects play a major
role, see \cite{Cox}, whereas inter-particle interactions in the
bulk are not essential \cite{K}. Thus, in order to understand
mathematical mechanisms of interaction-induced self-regulation, one
has to study the dynamics of an infinite population in a noncompact
habitat.

\subsection{Poisson states and self-regulation in infinite populations}
In this article, we propose and study a model of an infinite
population of point `particles' (migrants),  arriving in and
departing from  a habitat $X\subseteq \mathds{R}^d$, $d\geq 1$. Both
these acts of their dynamics occur at random and are described in a
way that takes into account possible heterogeneity of the habitat
and inter-particle attraction and repulsion (competition). It is
clear, cf. \cite[Sect. 2.3]{KK2}, that without interaction the
distribution of migrants should eventually reflect the heterogeneity
of the habitat and be eventually Poissonian \cite{Kingman} -- in
view of the randomness mentioned above and the lack of interactions.
It might also be clear that if the already existing population
attracts the newcomers, they can form dense clusters and thus their
distribution will no longer be Poissonian. Then the question which
we raise and answer here is whether mutual repulsion (competition)
-- which gives rise to the increase of emigration -- can prevent the
appearance of dense clusters and thus make the states nearly
Poissonian. If this is the case, one can qualify it as a kind of
self-regulation of the considered population of migrants --
similarly as in the case of birth-and-death \cite{KK1} and
fission-death \cite{KT} systems.

For convenience, including for having translation invariance, we
take the habitat $X=\mathds{R}^d$, equipped with the natural
(Euclidean) distance and the corresponding mathematical structures:
topology, Borel $\sigma$-field of subsets, etc. The (pure) states of
the population are \emph{configurations} -- locally finite subsets
$\gamma \subset \mathds{R}^d$. The latter means that, for a given
$\gamma$, the set $\gamma_\Lambda:=\gamma\cap\Lambda$ is finite
whenever $\Lambda \subset \mathds{R}^d$ is compact. By $\Gamma$ we
denote the collection of all such configurations $\gamma$. For each
compact $\Lambda\subset \mathds{R}^d$, one defines the counting map
$\Gamma \ni \gamma \mapsto N_\Lambda (\gamma) := |\gamma_\Lambda|$,
where $|\cdot|$ denotes cardinality. Thereby, one sets
$\Gamma^{\Lambda,n}:=\{ \gamma \in \Gamma : |\gamma_\Lambda| = n\}$,
$n \in \mathds{N}_0$, and equips $\Gamma$ with the $\sigma$-field
generated by all such $\Gamma^{\Lambda,n}$. This allows for
considering probability measures on $\Gamma$ as states of the
population -- a natural way of doing this in view of the randomness
of the basic evolution acts. Among such states are Poissonian ones.
In these states, the particles are independently distributed over
$\mathds{R}^d$, see \cite[Chapter 2]{Kingman}. They may serve as
reference states, with which other states are compared.

The homogeneous Poisson measure $\pi_\varkappa$ with density
$\varkappa
>0$ is defined by its values on $\Gamma^{\Lambda,n}$ with $n\in
\mathds{N}_0$ and compact $\Lambda$, given by the formula
\begin{equation}
  \label{J1}
\pi_\varkappa (\Gamma^{\Lambda,n}) = \left(\varkappa {\rm
V}(\Lambda)\right)^n \exp\left( - \varkappa {\rm V}(\Lambda)
\right)/ n!,
\end{equation}
where ${\rm V}(\Lambda)$ denotes Lebesgue's measure (volume) of
$\Lambda$. The probabilistic interpretation of (\ref{J1})  might be
as follows: $\Gamma^{\Lambda,n}$ is the event ``there is $n$
particles in $\Lambda$", and $\pi_\varkappa (\Gamma^{\Lambda,n})$ is
its probability if the state is $\pi_\varkappa$. The expected number
of particles in $\Lambda$ then is $\varkappa{\rm
V}(\Lambda)=$density$\times$volume. The independent character of the
distribution of particles is reflected in the fast decay of the
right-hand side of (\ref{J1}) with $n$. Then, for a state $\mu$, the
appearance of dense clusters of particles in a given $\Lambda$ in
this state can be established by the fact that the decay of
$\mu(\Gamma^{\Lambda,n})$ is not as fast as in (\ref{J1}), i.e., the
probability law of $N_\Lambda$ has a \emph{heavy tail}. In this
article, we employ \emph{sub-Poissonian} states (cf. Definition
\ref{J1df} and Remark \ref{I1rk} below), for each of which and for
every compact $\Lambda \subset \mathds{R}^d$, the following holds
\begin{equation}
  \label{J2}
\forall n\in \mathds{N}\qquad \mu(\Gamma^{\Lambda,n}) \leq n!
\left(\frac{e}{n}\right)^n \pi_\varkappa(\Gamma^{\Lambda,n}) ,
\end{equation}
with some positive $\varkappa$. In view of this bound,
sub-Poissonian states are characterized by the lack of
\emph{clustering} or, equivalently, by the lack of heavy tails of
the law of $N_\Lambda$ (following by (\ref{J2}) and Stirling's
formula). The particles in sub-Poissonian state are either
independent in taking their positions or `prefer' to stay away of
each other.

The counting map $\Gamma \ni \gamma \mapsto N_\Lambda (\gamma)$ can
also be defined for $\Lambda = \mathds{R}^d$. Set $\Gamma^n=
\{\gamma\in \Gamma: |\gamma|=n\}$. The set of \emph{finite}
configurations
\begin{equation}
 \label{J3a}
 \Gamma_0:= \bigcup_{n\in \mathds{N}_0} \Gamma^n
\end{equation}
is clearly a measurable subset of $\Gamma$. In a state, $\mu$, with
the property $\mu(\Gamma_0)=1$, the system  is \emph{finite}. Let us
show that $\pi_\varkappa(\Gamma_0) =0$. To this end, we take an
exhausting sequence $\{\Lambda_m\}_{m\in \mathds{N}}$ of compact
$\Lambda_m \subset X$. That is, $\Lambda_m \subset \Lambda_{m+1}$,
$m\in \mathds{N}$, and each $x\in X$ is eventually contained in its
elements. It is clear that ${\rm V}(\Lambda_m) \to +\infty$ in this
case. Moreover, for each $n$, it is possible to pick
$\{\Lambda_m\}_{m\in \mathds{N}}$ in such a way that
\begin{equation}
  \label{JJ3a}
 \sum_{m=1}^{\infty} \pi_\varkappa (\Gamma^{\Lambda_m,n}) < \infty,
\end{equation}
see (\ref{J1}). Assume now that $\pi_{\varkappa} (\Gamma_0)>0$. By
(\ref{J3a}) it then follows that $\pi_{\varkappa} (\Gamma^n)>0$ for
some $n\in \mathds{N}$. Fix this $n$ and write $\Gamma^n =
\Gamma^{\Lambda_m,n} \cup \Gamma_c^{\Lambda_m,n}$,
$\Gamma_c^{\Lambda_m,n}:= \Gamma^n \setminus \Gamma^{\Lambda_m,n}$,
which makes sense for each $m$. For a given $\varepsilon$, we pick
$m_\varepsilon$ such that ${\rm V}(\Lambda_{m_\varepsilon})
\varkappa \geq n$ and
\begin{equation}
  \label{JJ3b}
 \max\left\{ \pi_{\varkappa} (\Gamma^{\Lambda_{m_\varepsilon},n}); \sum_{k=m_\varepsilon+ 1}^{\infty} \pi_\varkappa (\Gamma^{\Lambda_k,n}) \right\}
 < \frac{\varepsilon}{2},
\end{equation}
which is possible in view of (\ref{J1}) and (\ref{JJ3a}). Clearly,
the letter inequality remains true after replacing $m_\varepsilon$
by any $m>m_\varepsilon$. By our assumptions each $\gamma\in
\Gamma^n$ is contained in some $\Lambda_m$, hence
$$\forall m>m_\varepsilon \qquad \Gamma_c^{\Lambda_m,n} \subset \bigcup_{k=m_\varepsilon+
1}^{\infty} \Gamma^{\Lambda_k,n},$$ by which and (\ref{JJ3b}) we
then conclude that $\pi_\varkappa (\Gamma^n) < \varepsilon$, which
in fact yields $\pi_\varkappa (\Gamma_0) =0$. Hence, the system in
state $\pi_\varkappa$ is infinite. A nonhomogeneous Poisson measure
$\pi_\varrho$, characterized by a density $\varrho:\mathds{R}^d\to
[0,+\infty)$, satisfies (\ref{J1}) with $\varkappa {\rm V}(\Lambda)$
replaced by $\int_\Lambda \varrho(x) dx$. Then either
$\pi_\varrho(\Gamma_0) =1$ or $\pi_\varrho(\Gamma_0) =0$, depending
on whether or not $\varrho$ is globally integrable. In this work, we
consider infinite systems.

\subsection{The model}

To characterize  states on $\Gamma$ one employs {\it observables} --
appropriate functions $F:\Gamma \rightarrow \mathds{R}$. Their
evolution is obtained from the backward Kolmogorov equation
\begin{equation}
 \label{R2}
\frac{d}{dt} F_t = L F_t , \qquad F_t|_{t=0} = F_0, \qquad t>0,
\end{equation}
where the Kolmogorov operator $L$ specifies the model. The states'
evolution is then obtained from the forward Kolmogorov (called also
Fokker--Planck) equation
\begin{equation}
 \label{R1}
\frac{d}{dt} \mu_t = L^* \mu_t, \qquad \mu_t|_{t=0} = \mu_0,
\end{equation}
related to that in (\ref{R2}) by the duality $\mu_t(F_0) =
\mu_0(F_t)$, where
\[
\mu(F) := \int_{\Gamma} F(\gamma) \mu(d \gamma).
\]
Direct solving of (\ref{R1}) assumes putting it in a Banach space
setting, which includes also defining $L^*$ as a linear operator in
the corresponding Banach space, see \cite{K}. For infinite
populations, this is usually impossible. Instead, one may solve the
corresponding weak Fokker-Planck equation
\begin{equation}
  \label{FPE}
  \mu_t (F) = \mu_0 (F) + \int_0^t \mu_s (LF) ds,
\end{equation}
that has to hold for all $F$ from a sufficiently big class of
functions. We refer the reader to the monograph \cite{Mich} for a
general theory of Fokker-Planck equations.

The model proposed and studied in this work is specified by the
following Kolmogorov operator
\begin{eqnarray}
 \label{L}
 \left(L F \right)(\gamma) & = &    \int_{\mathds{R}^d}   E^{+} (x, \gamma ) \left[F(\gamma \cup x) - F(\gamma) \right]dx\\[.2cm]
 & - & \nonumber \sum_{x\in \gamma} E^{-} (x, \gamma \setminus x \left[F(\gamma) - F(\gamma \setminus x)
 \right].
\end{eqnarray}
The first term in (\ref{L}) describes immigration with rate
\begin{equation}
 \label{J4}
E^{+} (x, \gamma ) = b^{+}(x) + \sum_{y\in \gamma} a^{+} (x-y),
\end{equation}
where $b^{+}$ and $a^{+}$ are nonnegative. Here the first term
corresponds to state-independent immigration, whereas the second one
describes attraction of the arriving `immigrants' by the existing
population. The second term in (\ref{L}) describes emigration.
Similarly as in (\ref{J4}), we take it in the form
\begin{equation}
 \label{J6}
E^{-} (x, \gamma ) =  b^{-} (x)+ \sum_{y\in \gamma} a^{-} (x-y),
\end{equation}
where both $b^{-}$ and $a^{-}$ are nonnegative. Note that the second
term in (\ref{J6}) describes repulsion of the particle located at
$x$ by the rest of the population, that can also be considered as
competition. Note also that the dependence of $b^{\pm}$ on $x\in X$
reflects possible heterogeneity of the habitat. In the sequel, the
functions $a^{\pm}$ and $b^{\pm}$ are called kernels. The
probabilistic meaning of both terms in (\ref{L}) is as follows. For
a given $\gamma$ and $x\in \gamma$, the probability to find a
particle at $x$ after time $t$ is $\exp(- t E^{-}(x, \gamma\setminus
x))$. Likewise, the probability to find a new particle at $y\in
\mathds{R}^d$ after time $t$ is $1-\exp(- t E^{+}(y, \gamma))$,
where the difference in these two formulas is related to the
different signs of the corresponding terms of $L$ and different
initial conditions.

In this article, the model parameters are supposed to satisfy the
following.
\begin{assumption}
  \label{Ass1}
The kernels $a^{\pm}$ in (\ref{J4}) and (\ref{J6}) are continuous
and belong to $L^1 (\mathds{R}^d) \cap L^\infty (\mathds{R}^d)$. The
kernels $b^{\pm}$ are continuous and bounded.
\end{assumption} According to this we set
\begin{gather}
  \label{J6a}
 \|a^{\pm}\| = \sup_{x\in \mathds{R}^d} a^{\pm}(x), \qquad
\|b^{\pm}\| = \sup_{x\in \mathds{R}^d} b^{\pm}(x).
\end{gather}
The assumed continuity has rather technical origin, whereas the
boundedness and integrability are essential.
\begin{remark}
  \label{CArk}
  Concerning the kernels $a^{\pm}$, the following alternatives are
possible:
\begin{itemize}
  \item[(i)] (\emph{long competition}) there exists $\theta>0$ such that $a^{-} (x) \geq \theta
  a^{+}(x)$ for all $x\in
  \mathds {R}^d$;
\item[(ii)] (\emph{short competition}) for each $\theta>0$, there exists
$x\in  \mathds {R}^d$ such that $a^{-} (x) < \theta
  a^{+}(x)$.
\end{itemize}
\end{remark}
In case (i), $a^{+}$ decays faster than $a^{-}$, and hence the
effect of repulsion from the existing population prevails. If $b^{+}
(x)\equiv 0$, new members appear only due to the existing
population, which can also be interpreted  as their birth. In this
case, $a^{+}$ is usually referred to as \emph{dispersal} kernel and
(i) then  corresponds to short dispersal. This particular case of
(\ref{L}) with nonzero $b^{-}$ and $a^{-}$ is the Bolker-Pacala
model -- introduced in \cite{BP1,BP3} and studied in
\cite{KK,KK1,Omel}. Such models with short dispersal are employed to
describe, e.g., the evolution of cell communities \cite{DimaRR}. An
instance of the long competition (short dispersal) is given by
$a^{+}$ with finite range, i.e., $a^{+}(x)\equiv 0$ for all $|x|
\geq r$, and $a^{-}(x)>0$ for such $x$. In case (ii), $a^{-}$ decays
faster than $a^{+}$, and hence some of the newcomers can be out of
the competitive influence of the existing population. Models of this
kind can be adequate, e.g., in plant ecology with the long-range
dispersal of seeds, cf. \cite{Ova}. Notably, the equality
$a^{+}=a^{-}$ -- corresponding to case (i) --  does not yet mean the
lack of interaction. This is a typical example of `nonlinearity' of
a complex system mentioned above.

Particular cases of the model specified by (\ref{L}) were studied
in: (a) \cite{FKKK,KK,KK1}, case of $b^{+}\equiv 0$; (b) \cite{KK2},
case of $a^{+}\equiv 0$. In case (a), our result -- formulated in
Theorem \ref{1tm} below -- yields an extension of the corresponding
result of \cite{KK} as it holds for both cases, (i) and (ii),
mentioned in Remark \ref{CArk}.

In Section 2 below, we introduce necessary technicalities and then
formulate the result: Theorem \ref{1tm} and Corollary \ref{Marco}.
Thereafter, we make a number of comments to these statements and
compare them with the facts known for similar models. We also
discuss possible continuation of studying the proposed model based
on the result stated in Theorem \ref{1tm}. A complete proof of the
latter will be done in a separate publication.

\section{The Result}

By $\mathcal{B}(\mathds{R}^d)$ we denote the sets of all Borel
subsets of $\mathds{R}^d$. The configuration space $\Gamma$ is
equipped with the vague (weak-hash) topology, see \cite{DV1}, and
thus with the corresponding Borel $\sigma$-field
$\mathcal{B}(\Gamma)$, which makes it a standard Borel space. Note
that $\mathcal{B}(\Gamma)$ is exactly the $\sigma$-field generated
by the sets $\Gamma^{\Lambda,n}$ mentioned in Introduction.  By
$\mathcal{P}(\Gamma)$ we denote the set of all probability measures
on $(\Gamma, \mathcal{B}(\Gamma))$.

\subsection{Sub-Poissonian measures}
The space of finite configurations $\Gamma_0$ defined in (\ref{J3a})
can be equipped with the topology induced by the vague topology of
$\Gamma$. It is precisely the weak topology determined by bounded
continuous functions $f\in C_{\rm b}(\mathds{R}^d)$. Then the
corresponding Borel $\sigma$-field $\mathcal{B}(\Gamma_0)$ is a
sub-field of $\mathcal{B}(\Gamma)$. One can show that a function
$G:\Gamma_0 \to \mathds{R}$ is measurable if and only if there
exists a family of symmetric Borel functions $G^{(n)}:
(\mathds{R}^d)^n \to \mathds{R}$, $n\in \mathds{N}$ such that
\begin{equation}
  \label{C22b}
G(\{x_1 , \dots , x_n\}) =   G^{(n)}(x_1 , \dots , x_n).
\end{equation}
We also set $G^{(0)}= G(\varnothing)$ and consider the following
class of functions.
\begin{definition}
  \label{Bbsdf}
A measurable function, $G:\Gamma_0\to \mathds{R}$, is said to have
bounded support if there exist $N\in \mathds{N}$ and a compact
$\Lambda$ such that: (a) $G^{(n)}\equiv 0$ for all $n>N$; (b)
$G(\eta)=0$ whenever $\eta$ is not a subset of $\Lambda$. By $B_{\rm
bs}$ we will denote the set of all bounded functions with bounded
support. For $G\in B_{\rm bs}$,  $N_G$ and $\Lambda_G$ will denote
the least $N$ and $\Lambda$ as in (a) and (b), respectively. We also
set $C_G=\sup_{\eta\in \Gamma_0} |G(\eta)|$.
\end{definition}
For $G\in B_{\rm bs}$, set
\begin{equation}
  \label{A2}
  (KG)(\gamma) = \sum_{\eta \subset \gamma} G(\eta), \qquad
  \gamma\in \Gamma.
\end{equation}
\begin{remark}
  \label{Bbsrk}
For each $G\in B_{\rm bs}$, $KG$ is measurable and such that
$|(KG)(\gamma)| \leq C_G ( 1 + |\gamma \cap \Lambda_G|^{N_G})$ with
$C_G$, $\Lambda_G$ and $N_G$ as in Definition \ref{Bbsdf}.
\end{remark}
The Lebesgue-Poisson measure $\lambda$ is defined on $\Gamma_0$ by
the integrals
\begin{equation}
 \label{C22c}
\int_{\Gamma_0} G(\eta) \lambda(d \eta) = G(\varnothing ) +
\sum_{n=1}^\infty \frac{1}{n!} \int_{(\mathds{R}^d)^n} G^{(n)} (x_1
, \dots, x_n) d x_1 \cdots d x_n,
\end{equation}
holding for all $G\in B_{\rm bs}$, for which $G^{(n)}$'s are defined
as in  (\ref{C22b}) .

Let $\mathcal{P}_{\rm lm}(\Gamma)$ be the set of all probability
measures on $\Gamma$ that have all local moments. This means that
such measures satisfy the condition
\begin{equation*}
 \int_{\Gamma} |\gamma_\Lambda|^n \mu (d \gamma) < \infty,
\end{equation*}
holding for all $n\in \mathds{N}$ and compact $\Lambda \subset
\mathds{R}^d$. Let $\varTheta$ be the set of all continuous
compactly supported functions $\theta: \mathds{R}^d \to (-1, 0]$.
Then, for each $\mu \in \mathcal{P}(\Gamma)$, the function
\begin{equation*}
  F^\theta (\gamma) := \prod_{x\in \gamma} ( 1 + \theta (x)) = \exp\left(  \sum_{x\in \gamma} \log (1 + \theta(x)) \right), \qquad
  \gamma \in \Gamma ,
\end{equation*}
is $\mu$-integrable. For $\theta \in \varTheta$, the map $\gamma
\mapsto \sum_{x\in \gamma} \log (1 + \theta(x))$ -- and hence the
function $F^\theta$ -- are continuous in the vague topology. It is
clear that each $\theta \in \varTheta$ is absolutely integrable on
$X$; hence, one can write
\begin{equation}
  \label{norm}
  \|\theta\|= \int_X |\theta(x)| d x.
\end{equation}
For $\theta \in \varTheta$, we define functions $\Gamma_0 \ni \eta
\mapsto e(\theta;\eta)$, $\Gamma_0 \ni \eta \mapsto
e_n(\theta;\eta)$, $n\in \mathds{N}_0$, by setting
\begin{equation}
  \label{lm1}
  e(\theta;\eta) = \prod_{x\in \eta} \theta(x), \qquad e_n (\theta;\eta) = e(\theta;\eta)
  \mathds{1}_{\Gamma^n}(\eta),
\end{equation}
where $\mathds{1}_{\Gamma^n}$ is the corresponding indicator. For
$n=0$, the second term in (\ref{lm1}) has to be understood by taking
into account that
\[
\Gamma^0 = \{\varnothing\} \qquad {\rm and} \qquad \prod_{x\in
\varnothing} \theta(x) =1.
\]
Clearly, $e_n (\theta;\cdot)\in B_{\rm bs}$; hence, $K
e_n(\theta;\cdot)$ satisfies the corresponding estimate as in Remark
\ref{Bbsrk}, which implies that $Ke_n (\theta;\cdot)$ is
$\mu$-integrable for all $\mu \in \mathcal{P}_{\rm lm}(\Gamma)$.
\begin{definition}
  \label{J1df}
The set of states $\mathcal{P}_{\rm exp}(\Gamma)$ contains all those
$\mu \in \mathcal{P}_{\rm lm}(\Gamma)$, for each of which, the map
\begin{equation}
  \label{SR1}
\varTheta \ni \theta \mapsto  \mu(K e_n(\theta;\cdot)) =
\int_{\Gamma} (K e_n (\eta;\cdot))(\gamma) \mu ( d \gamma)
\end{equation}
can be extended to a continuous monomial of order $n$ on the real
Banach space $L^1 (X)$ satisfying the estimate
\begin{equation}
  \label{lm2}
\left|\mu(K e_n(\theta;\cdot))\right| \leq \frac{1}{n!} \varkappa^n
\|\theta\|^n,
\end{equation}
for a certain $\varkappa>0$ and $\|\cdot\|$ as in (\ref{norm}).
\end{definition}
It can be shown that, for each $\mu \in \mathcal{P}_{\rm
exp}(\Gamma)$, the right-hand side of (\ref{SR1}) might be written
in the form
\begin{equation}
  \label{Mar}
\mu(K e_n(\theta;\cdot)) = \frac{1}{n!} \int_{X^n} k_\mu^{(n)} (x_1,
\dots , x_n) \theta(x_1) \cdots \theta(x_n)d x_1 \cdots d x_n,
\end{equation}
where $k_\mu^{(n)}$ -- $n$-th order correlation function of $\mu$ --
is a symmetric elements of $L^\infty (X^n)$ satisfying the following
\begin{equation}
  \label{Mar1}
  0 \leq k_\mu^{(n)} (x_1,
\dots , x_n) \leq \varkappa^n,
\end{equation}
with $\varkappa$ being the same as in (\ref{lm2}). The right-hand
inequality in (\ref{Mar1}) is known as Ruelle's bound. Note that
$k_\mu^{(0)} (x_1, \dots , x_n) \equiv 1$ since $\mu(\Gamma)=1$. For
the homogeneous Poisson measure, it follows that
\begin{equation}
  \label{Mar2}
k_{\pi_\varkappa}^{(n)} (x_1, \dots , x_n) = \varkappa^n, \qquad
n\in \mathds{N}_0.
\end{equation}
For a compact $\Lambda \subset X$, let $\mathds{1}_\Lambda$ be its
indicator. Then one can write
\[N_\Lambda (\gamma) =
|\gamma\cap\Lambda| = \sum_{x\in \gamma} \mathds{1}_\Lambda (x),
\]
which can be generalized to the following, cf. (\ref{A2}),
\begin{eqnarray}
  \label{Mar3}
  N_\Lambda^n (\gamma) & = & \sum_{l=1}^n l! S(n,l) \sum_{\{x_1, \dots ,
  x_l \}\subset \gamma} \mathds{1}_\Lambda (x_l) \cdots \mathds{1}_\Lambda
  (x_l) \\[.2cm] \nonumber & =& \sum_{l=1}^n l! S(n,l) ( Ke_l) (\mathds{1}_\Lambda;\cdot)(\gamma),
\end{eqnarray}
where $S(n.l)$ is Stirling's number of second kind -- the number of
ways to divide $n$ labeled items into $l$ unlabeled groups, see
\cite[Chapter 2]{Riordan}. Now we apply (\ref{Mar}) and (\ref{Mar1})
in (\ref{Mar3}) and obtain
\begin{eqnarray}
  \label{Mar4}
  \mu(N_\Lambda^n) & = & \sum_{l=1}^n S(n,l) \int_{\Lambda^l} k^{(l)}_\mu
  (x_1 , \dots , x_l) d x_1 \cdots d x_l \\[.2cm] \nonumber & \leq & \sum_{l=1}^n S(n,l)
  \left( \varkappa {\rm V}(\Lambda)\right)^l =
  \pi_\varkappa (N_\Lambda^n) , \qquad n\in
  \mathds{N}.
\end{eqnarray}
By the latter equality it follows that
\[
\pi_\varkappa (N_\Lambda^n) = T_n (\varkappa {\rm V}(\Lambda)),
\]
where $T_n$ is Touchard's polynomial, cf. \cite[Chapter 2]{Riordan}.
Then
\begin{eqnarray}
  \label{Mar4a}
 \int_\Gamma \exp\left( \beta N_\Lambda (\gamma)\right)\pi_\varkappa (d \gamma) &
 = &
 \sum_{n=0}^\infty \frac{\beta^n}{n!} T_n (\varkappa {\rm
 V}(\Lambda)) \\[.2cm]\nonumber & = & \exp\left(\varkappa {\rm
 V}(\Lambda) (e^\beta -1)\right), \qquad \beta \in \mathds{R}.
\end{eqnarray}
At the same time, for each $\mu \in \mathcal{P}_{\rm exp}(\Gamma)$,
we have
\begin{equation}
  \label{Mar5}
  \int_{\Gamma} \exp \left(\beta N_\Lambda (\gamma) \right) \mu( d
  \gamma) = \sum_{n=0}^\infty e^{\beta n} \mu(\Gamma^{\Lambda,n}),
\end{equation}
which for $\pi_\varkappa$ reads
\begin{equation}
  \label{Mar6}
\int_{\Gamma} \exp \left(\beta N_\Lambda (\gamma) \right)
\pi_\varkappa ( d\gamma) = \sum_{n=0}^\infty e^{\beta n}
\pi_\varkappa(\Gamma^{\Lambda,n})
\end{equation}
Now we take into account (\ref{Mar4}), set in (\ref{Mar5}) and
(\ref{Mar6}) $t = e^\beta$, and obtain from (\ref{Mar4a}) the
following
\[
\sum_{n=0}^\infty t^n \mu(\Gamma^{\Lambda,n}) \leq \sum_{n=0}^\infty
t^n \pi_\varkappa(\Gamma^{\Lambda,n}) = e^{- \varkappa {\rm
V}(\Lambda)} \exp\left( t\varkappa {\rm V}(\Lambda)\right).
\]
Then (\ref{J2}) is obtained from the latter by a standard
Cauchy-like estimation.
\begin{remark}
  \label{I1rk}
Each $\mu\in \mathcal{P}_{\rm exp}(\Gamma)$ is sub-Poissonian in the
sense of (\ref{J2}) and (\ref{Mar4}). Furthermore, let $G\in B_{\rm
bs}$ be positive, i.e., $G(\eta)\geq 0$ for $\lambda$-almost all
$\eta\in \Gamma_0$, see (\ref{C22c}). Then
\[
\mu(K G) \leq \pi_\varkappa (KG) = G(\varnothing) +
\sum_{n=1}^\infty \frac{\varkappa^n}{n!}\int_{X^n} G^{(n)}(x_1 ,
\dots x_n) d x_1 \cdots d x_n.
\]
\end{remark}
The proof of the latter estimate is based on (\ref{Mar1}) and the
fact that
\begin{eqnarray}
  \label{Ma}
\mu(K G) & = & G(\varnothing) + \sum_{n=1}^\infty\frac{1}{n!}
\int_{X^n} k_\mu^{(n)} (x_1 , \dots  x_n) G^{(n)}(x_1 , \dots  x_n)
d x_1
\cdots d x_n \qquad  \\[.2cm] \nonumber
& = & \int_{\Gamma_0} k_\mu(\eta) G(\eta) \lambda ( d \eta),
\end{eqnarray}
holding for all $\mu\in \mathcal{P}_{\rm exp}(\Gamma)$. In the
second line of (\ref{Ma}) we use the \emph{correlation function} of
$\mu$, related to $k_\mu^{(n)}$, $n\in \mathds{N}_0$ in the sense of
(\ref{C22b}).

By (\ref{lm1}) we have
\[
e(\theta;\eta) = \sum_{n=0}^\infty e_n (\theta;\eta),
\]
which implies, see also (\ref{lm2}), that
\begin{equation}
  \label{I1}
\mu (Ke(\theta;\cdot)) = \mu(F^\theta)= \int_{\Gamma} \prod_{x\in
\gamma} ( 1 + \theta (x)) \mu( d \gamma).
\end{equation}
This means that $\varTheta \ni \mapsto \mu(F^\theta)$ can be
extended to a real exponential type entire function of $\theta\in
L^1(X)$. For the homogeneous Poisson measure, the integral in
(\ref{I1}) can be calculated explicitly, cf. (\ref{Mar2}),
\begin{eqnarray}
  \label{I2}
\pi_\varkappa (F^\theta) & = & 1+ \sum_{n=1}^\infty \frac{1}{n!}
\int_{X^n} k_{\pi_\varkappa}^{(n)} (x_1 , \dots , x_n)
\theta(x_1)\cdots \theta(x_n) d x_1 \cdots d x_n
\qquad \\[.2cm] \nonumber & = &\exp\left(\varkappa \int_{\mathbb{R}^d}\theta (x) d x \right).
\end{eqnarray}

\subsection{The statement}

Before formulating the result we define the set of functions
$F:\Gamma\to \mathds{R}$ which we then use in (\ref{FPE}).
\begin{definition}
  \label{Mardf}
The set $\mathcal{G}$ of functions $G:\Gamma_0 \to \mathds{R}$ is
defined as that containing all those $G$ which satisfy
\[
\forall C>0 \qquad \int_{\Gamma_0} C^{|\eta|} \left| G(\eta) \right|
\lambda ( d \eta) < \infty.
\]
\end{definition}
Note that
\[
\int_{\Gamma_0} C^{|\eta|} \left|e(\theta;\eta) \right| \lambda ( d
\eta) = 1 +\sum_{n=1}^\infty \frac{1}{n!} \left( C
\|\theta\|\right)^n <\infty.
\]
Hence, $e(\theta;\cdot) \in \mathcal{G}$ if $\theta$ is integrable.
\begin{proposition}
 \label{Marpn}
For each $G\in \mathcal{G}$, the function $\Gamma \ni \gamma \mapsto
(K G)(\gamma)$ is $\mu$-integrable whenever $\mu$ is in
$\mathcal{P}_{\rm exp}(\Gamma)$.
\end{proposition}
\begin{proof}
For $G\in \mathcal{G}$, compact $\Lambda$ and $N\in \mathds{N}$, set
\[
G^{\Lambda,N}(\eta) =\left\{ \begin{array}{ll} G(\eta), \qquad &{\rm
} \eta \subset \Lambda  \ {\rm and} \ |\eta|\leq N; \\[.3cm] 0,
\qquad &{\rm otherwise}.
\end{array} \right.
\]
Then $G^{\Lambda,N} \in B_{\rm bs}$ and hence, see (\ref{Ma}) and
(\ref{Mar2}),
\begin{eqnarray*}
 \mu \left( K G^{\Lambda,N} \right) & = & G(\varnothing) +\sum_{n=1}^N
 \int_{\Lambda^n} k^{(n)}_\mu (x_1, \dots , x_n) G^{(n)}(x_1, \dots
 , x_n) d x_1 \cdots d x_n \qquad \qquad \\[.2cm]\nonumber & \leq &
 |G(\varnothing)| + \sum_{n=1}^\infty \frac{\varkappa^n}{n!}
 \int_{X^ n} |G^{(n)}(x_1, \dots
 , x_n)| d x_1 \cdots d x_n < \infty.
\end{eqnarray*}
Then the stated property follows by Lebesgue's dominated convergence
theorem.
\end{proof}
By Proposition \ref{Marpn} it follows that, for $\mu\in
\mathcal{P}_{\rm exp}(\Gamma)$ and $G\in \mathcal{G}$. the following
holds
\begin{eqnarray}
  \label{Mar7a}
  \mu(KG) = \int_{\Gamma_0} k_\mu (\eta) G(\eta) \lambda (d \eta),
\end{eqnarray}
see (\ref{Ma}). Now we set
\begin{equation}
  \label{Mar8}
 \mathcal{F}=\{ F= KG: G\in \mathcal{G}\}.
\end{equation}
\begin{proposition}
  \label{Mapn}
For each $F\in \mathcal{F}$, it follows that $L F\in \mathcal{F}$.
\end{proposition}
\begin{proof}
For $G\in B_{\rm bs}$, se define, cf. (\ref{L}),
\begin{eqnarray}
  \label{Mar9}
(\widehat{L}G)(\eta)& = & \int_X E^{+} (x,\eta) G(\eta\cup x) d x +
\sum_{x\in \eta} \int_X a^{+} (x-y) G(\eta\setminus x\cup y) d y
\qquad \qquad
\\[.2cm] \nonumber & - & \left( \sum_{x\in \eta} E^{-} (x, \eta \setminus
x)\right) G(\eta) - \sum_{x\in \eta} \left( \sum_{y\in \eta\setminus
x} a^{-}(x-y ) \right) G(\eta\setminus x) \\[.2cm] \nonumber & =: &
(A_1 G)(\eta) +\cdots + (A_4 G)(\eta),
\end{eqnarray}
that ought to hold for $\lambda$-almost all $\eta\in \Gamma_0$. Here
$E^{\pm}$ are given in (\ref{J4}) and (\ref{J6}). It is clear that
$|(\widehat{L}G)(\eta)|< \infty$ as the sums in (\ref{Mar9}) are
finite and the integrals are taken over a compact subset of $X$. To
prove that $LF \in \mathcal{F}$ we have to show that: (a)
$\widehat{L}$ can be extended to a self-map of $\mathcal{G}$; (b)
this extension and the operator defined in (\ref{L}) satisfy $L KG =
K \widehat{L} G$ holding for all $G\in \mathcal{G}$. To get (a) we
employ the following evident property of the integrals defined in
(\ref{C22c})
\begin{equation}
\label{I4} \int_{\Gamma_0} \left(\int_X H(x , \eta) d x \right)
\lambda ( d \eta) = \int_{\Gamma_0} \left( \sum_{x\in \eta} H(x ,
\eta\setminus x)\right) \lambda (d \eta),
\end{equation}
holding for all appropriate functions $H$. Along with (\ref{I4}) we
use the estimate, see (\ref{J6a}),
\begin{equation}
\label{I4a} \left|\sum_{x\in \eta} E^{\pm} (x, \eta) \right| \leq
|\eta| \left( \|b^{\pm} \| + \frac{1}{2} \| a^{\pm}\|
(|\eta|-1)\right)=: c^{\pm} (|\eta|), \qquad \eta \in \Gamma_0.
\end{equation}
For $C>0$, set $\mathcal{G}_C = L^1 ( \Gamma_0, C^{|\cdot|} d
\lambda)$, and let $\|\cdot\|_C$ stand for the corresponding norm.
Then (a) will be done if we show that each $A_i$ defined in the last
line of (\ref{Mar9}) acts as a bounded linear operator from
$\mathcal{G}_C$ to $\mathcal{G}_{C+\varepsilon}$, that holds for
each positive $C$ and $\varepsilon$. By (\ref{I4a}) we get
\begin{equation*}
 \sup_{n\in \mathds{N}} c^{\pm} (n) \left(\frac{ C}{C+\varepsilon} \right)^n
 = : \delta^{\pm} (C, \varepsilon) <\infty.
\end{equation*}
Then by means of (\ref{I4}) it follows that
\begin{eqnarray}
  \label{1fa}
 \|A_1 G\|_C & \leq & \int_{\Gamma_0} \left(\int_{X} \left|E^{+} ( x ,\eta)
 \right| |G(\eta\cup x) | d x\right) C^{|\eta|} \lambda ( d \eta)
\\[.2cm] \nonumber & = &
 \int_{\Gamma_0}  \left| \sum_{x\in \eta} E^{+}(x, \eta\setminus x)\right| |G(\eta)| C^{|\eta|-1} \lambda ( d \eta)\\[.2cm]
 & \leq  & C^{-1} \delta^{+} (C, \varepsilon) \|G\|_{C+\varepsilon}. \nonumber
\end{eqnarray}
In a similar way, we get
\begin{eqnarray*}
 \|A_2 G\|_C &\leq & \delta^{+} (C, \varepsilon) \|G\|_{C+\varepsilon},
 \\[.2cm] \nonumber \|A_3 G\|_C &\leq & \delta^{-} (C, \varepsilon)
 \|G\|_{C+\varepsilon}, \\[.2cm] \nonumber \|A_4 G\|_C & \leq & C \delta^{-} (C, \varepsilon)
 \|G\|_{C+\varepsilon}.
 \end{eqnarray*}
In combination with (\ref{1fa}) this implies that $\widehat{L}$ acts
as a bounded linear operator from each $\mathcal{G}_C$ to
$\mathcal{G}_{C+\varepsilon}$, which yields (a) since
\[
\mathcal{G}= \bigcap_{C>0} \mathcal{G}_C.
\]
For $G \in B_{\rm bs}$, the equality $L KG = K \widehat{L} G$
follows by \cite[Proposition 3.1, page 209]{Dima0}. Its extension to
$G\in \mathcal{G}$ follows as in (a).
\end{proof}
Now we are at a position to formulate our main statement.
\begin{definition}
  \label{Mardf}
By a solution of the Fokker-Planck equation (\ref{FPE}) we
understand a map $[0,+\infty) \ni t \mapsto \mu_t\in
\mathcal{P}(\Gamma)$ such that: (a) each $F\in \mathcal{F}$, see
(\ref{Mar8}), is $\mu_t$-absolutely integrable for all $t\geq 0$;
(b) for each $F\in \mathcal{F}$, the map $[0,+\infty) \ni t \mapsto
\mu_t(LF)\in \mathds{R}$ is integrable on each $[0,T]$, $T>0$ and
(\ref{FPE}) is satisfied.
\end{definition}
According to this definition, if $t\mapsto \mu_t$ is a solution,
then $t\mapsto \mu_t (F)\in \mathds{R}$ is absolutely continuous,
and hence $\mu_t (F) \to \mu_0 (F)$ as $t\to 0$, where $\mu_0$ is
considered as the initial condition for (\ref{FPE}).
\begin{theorem}
  \label{1tm}
For each $\mu_0 \in \mathcal{P}_{\rm exp}(\Gamma)$, there exists a
unique map $[0,+\infty) \ni t \mapsto \mu_t\in \mathcal{P}_{\rm
exp}(\Gamma)$ that solves (\ref{FPE}).
\end{theorem}
The proof of this statement is quite technical and will be done in a
separate publication. It consists in constructing a map $[0,+\infty)
\ni t \mapsto k_t$ such that: (a) each $k_t$ is the correlation
function of a unique $\mu_t\in \mathcal{P}_{\rm exp}(\Gamma)$, i.e.,
$k_t = k_{\mu_t}$, see (\ref{Mar}), (\ref{I1}); (b) the map
$[0,+\infty) \ni t \mapsto \mu_t\in \mathcal{P}_{\rm exp}(\Gamma)$
obtained in this way solves (\ref{FPE}). Note that, for each $F\in
\mathcal{F}$, the mentioned map $t \mapsto k_t$ is such that the
map, see (\ref{Mar7a}),
\begin{eqnarray*}
t \mapsto \mu_t(LF)=\mu_t (LK G) = \mu_t (K \widehat{L} G)   =
\int_{\Gamma_0} k_{t} (\eta) (\widehat{L} G) (\eta) \lambda ( d
\eta)
\end{eqnarray*}
is continuous and integrable on each $[0,T]$, $T>0$.
\begin{corollary}
  \label{Marco}
If the initial state of the population is sub-Poissonian, i.e.,
$\mu_0 \in \mathcal{P}_{\rm exp}(\Gamma)$, then the evolution of its
states $\mu_0\to \mu_t$ preserves this property. This is true for
both long and short competition cases mentioned in Remark
\ref{CArk}.
\end{corollary}
Since the solution described in Theorem \ref{1tm} is unique, the
evolution mentioned in Corollary \ref{Marco} is the only possible,
and hence the dynamics of the population manifests
competition-caused self-regulation.

\subsection{Further comments and comparison}

For $b^{+} \equiv 0$ and $a^{-} \equiv 0$, the model considered in
this work gets exactly soluble. It is known as the continuum contact
model, see \cite{Dima}, for which there is no self-regulation.
Namely, under the following quite natural condition
\begin{equation*}
\inf_{x,y\in B} a^{+}(x-y) \geq \alpha>0,
\end{equation*}
satisfied in some ball $B$, it was proved \cite[Eq. (3.5), page
303]{Dima} that
\[
k_{\mu_t}^{(n)} (x_1, \dots, x_n) \geq \omega^{n t} n!,
\]
holding for some $\omega>0$, all $n\geq 2$ and $t>0$, and almost all
$x_1 , \dots, x_n \in B$. Hence, for $\theta(x) >0$, $x\in B$, by
(\ref{Mar}) and (\ref{SR1}) the latter implies
\[
\mu_t (K e_n(\theta;\cdot)) \geq \omega^{n t},
\]
that clearly contradicts (\ref{lm2}). It is possible to show that a
similar bound holds true also in the model described by $L$ as in
(\ref{L}) with $a^{-} \equiv 0$ and nonzero $b^{+}$. This means that
the inter-particle competition represented in $L$ by $a^{-}$ -- that
gives rise to the increase of emigration -- is the sole factor
responsible for the effect mentioned in Corollary \ref{Marco}.
Likewise, by comparing with the Bolker-Pacala model obtained from
the letter by setting $b^{+}\equiv 0$, one shows that if the
following holds
 $$\inf_{x\in X} b^{-} (x)  \geq \int_{X} a^{+} (x) d x,$$ then the correlation
functions $k_{\mu_t}$ remain bounded in time, see \cite{KK}. That
is, the global regulation may be achieved at the expense of large
emigration. Moreover, the system eventually gets empty in this case.

As mentioned above, Poissonian states are completely characterized
by their densities. That is, for $\theta \in L^1(X)$ and a Poisson
state $\pi_\varrho$, $\varrho\in L^\infty (X)$, one has, cf.
(\ref{I2}),
\[
\pi_\varrho (F^\theta) = \exp\left( \int_X \varrho (x) \theta (x) dx
\right),
\]
which means  that the corresponding correlation functions are
\[
k_{\pi_\varrho}^{(n)} (x_1 , \dots , x_n) = \varrho(x_1) \cdots
\varrho(x_n), \qquad n\in \mathds{N},
\]
and hence $\varrho (x) = k^{(1)}_{\pi_\varrho}(x)$. At the
mesoscopic level obtained by a scaling procedure, see
\cite{FKKK,Koz}, the evolution of states is described as the
evolution of their densities $k_0^{(1)} \to k_t^{(1)}$ obtained by
solving corresponding kinetic equations.  Without interactions, the
description of the evolution of Poisson states $\pi_{\varrho_0}\to
\pi_{\varrho_t}$ by equations like (\ref{FPE}) is equivalent to the
that obtained from the corresponding kinetic equations. Possible
interactions in the system are taken into account in kinetic
equations indirectly, and hence the mesoscopic description is less
accurate. At the same time, by kinetic equations -- and their more
sophisticated versions \cite{Omel} -- it is possible to get much
richer information as to the evolution of a given system, see
\cite{Omel} for a numerical study of the Bolker-Pacala model. It is
believed that, for sub-Poissonian states, passing from micro- to
meso-scale -- and hence from equations like (\ref{FPE}) to kinetic
equations -- produces `minor errors', and hence is acceptable. This
means that the sub-Poissonicity established in Theorem \ref{1tm}
`justifies' passing to the description of the evolution of the
considered system based on the kinetic equation
\begin{eqnarray}
  \label{KE}
\frac{d}{dt} \varrho_t (x) & = & \left( b^{+} (x) - b^{-} (x)
\right) \varrho_t(x) \\[.2cm] \nonumber  & + & \int_X a^{+} (x-y) \varrho_t(y) d y -
\varrho_t(x) \int_X a^{-} (x-y) \varrho_t (y) dy,
\end{eqnarray}
which can be derived from (\ref{R2}) and (\ref{L}) in the same way
as it was done for the Bolker-Pacala model in \cite{FKKK}, and for a
similar migration model in \cite{Koz}. We believe that the study of
(\ref{KE}) can yield further details of the evolution of the model
proposed here -- similarly as it was in the case for the models
studied in \cite{FKKK,Koz}. In particular, we expect to clarify the
peculiarities of the cases mentioned in Remark \ref{CArk}. Note that
the self-regulation described in Theorem \ref{1tm} --  a global
effect -- occurs even if the competition kernel is very short. We
plan to study (\ref{KE}), also numerically, in a separate
publication.

\section*{Acknowledgment}
The present research was supported by National Science Centre,
Poland, grant 2017/25/B/ST1/00051, that is cordially acknowledged by
the author.

\end{document}